\documentclass[a4paper,10pt]{amsart}
\usepackage{amsmath, amssymb}

\usepackage[pagebackref]{hyperref}

\newcommand{\tbt}[4]{\begin{pmatrix}#1 & #2 \\ #3 & #4\end{pmatrix}}
\newcommand{\stbt}[4]{\left(\begin{smallmatrix}#1 & #2 \\ #3 & #4\end{smallmatrix}\right)}

\title[Adelic Galois representations]{Images of adelic Galois representations for modular forms}

\author{David Loeffler}
\address{Mathematics Institute\\
Zeeman Building, University of Warwick\\
Coventry CV4 7AL, UK}
\email{d.a.loeffler@warwick.ac.uk}

\thanks{The author's research was supported by the Royal Society University Research Fellowship ``$L$-functions and Iwasawa theory''. This work was prepared during a visit to the Mathematical Sciences Research Institute in Berkeley, California, supported by NSF Grant No. 0932078000.}

\theoremstyle{plain}
    \newtheorem{theorem}{Theorem}[subsection]
    \newtheorem{lemma}[theorem]{Lemma}
    \newtheorem{proposition}[theorem]{Proposition}
    \newtheorem{corollary}[theorem]{Corollary}
\theoremstyle{definition}
    \newtheorem{definition}[theorem]{Definition}
\theoremstyle{remark}
    \newtheorem{remark}[theorem]{Remark}
\newtheorem*{hypothesis}{Hypothesis}

 \DeclareMathOperator{\GL}{GL}
 \DeclareMathOperator{\PGL}{PGL}
 \DeclareMathOperator{\SL}{SL}
 \DeclareMathOperator{\PSL}{PSL}
 
 \DeclareMathOperator{\norm}{norm}
 \DeclareMathOperator{\Gal}{Gal}
 \DeclareMathOperator{\Ind}{Ind}
 \DeclareMathOperator{\Hyp}{Hyp}

 \newcommand{\cO}{\mathcal{O}}
 \newcommand{\frp}{\mathfrak{p}}
 
 \newcommand{\CC}{\mathbf{C}}
 \newcommand{\QQ}{\mathbf{Q}}
 \newcommand{\ZZ}{\mathbf{Z}}
 \newcommand{\FF}{\mathbf{F}}

 \newcommand{\Qp}{{\QQ_p}}
 \newcommand{\Zp}{{\ZZ_p}}

 \newcommand{\into}{\hookrightarrow}
\numberwithin{equation}{subsection}

\begin{document}

 \begin{abstract}
  We show that the image of the adelic Galois representation attached to a non-CM modular form is open in the adelic points of a suitable algebraic subgroup of $\GL_2$ (defined by F.~Momose). We also show a similar result for the adelic Galois representation attached to a finite set of modular forms.
 \end{abstract}

 \maketitle

 \section*{Introduction}
 
  Let $E$ be an elliptic curve over $\QQ$, and $p$ a prime number. Then the action of the Galois group on the Tate module of $E$ determines a Galois representation
  \[ 
   \rho_{E, p}: \operatorname{Gal}(\overline{\QQ} / \QQ) \to \GL_2(\ZZ_p). 
  \]
  If $\hat\ZZ = \prod_{p} \ZZ_p$ is the profinite completion of $\ZZ$ (the integer ring of the ring $\hat\QQ$ of finite adeles), then the product of the $\rho_{E, p}$ defines an adelic Galois representation
  \[ \rho_E:  \operatorname{Gal}(\overline{\QQ} / \QQ) \to \GL_2(\hat\ZZ).\]
  
  Suppose $E$ does not have complex multiplication. Then the images of these representations are described by the following three theorems, all of which are due to Serre:
  \begin{enumerate}
   \item[(A)] \cite[\S IV.2.2]{serre68b} For all primes $p$, the image of $\rho_{E, p}$ is open in $\GL_2(\Zp)$.
   \item[(B)] \cite[Theorem 2]{serre72} For all but finitely many $p$, the image of $\rho_{E, p}$ is the whole of $\GL_2(\Zp)$.
   \item[(C)] \cite[Theorem 3]{serre72} The image of the product representation $\rho_E$ is open in $\GL_2(\hat\ZZ)$.
  \end{enumerate}
  
  Note that (C) implies both (A) and (B), but the converse is not automatic; theorem (C) shows that not only do the $\rho_{E, p}$ individually have large image, but they are in some sense ``independent of each other'' up to a finite error.
  
  If one replaces the elliptic curve $E$ by a modular eigenform $f$, then one has $p$-adic Galois representations $\rho_{f, p}$ and an adelic representation $\rho_f$, and it is natural to ask whether analogues of theorems (A)--(C) hold in this context. For modular forms of level 1, analogues of all three theorems were obtained by Ribet \cite{ribet75}; but the case of modular forms of higher level is considerably more involved, owing to the presence of so-called ``inner twists''. 
  
  The appropriate analogues of (A) and (B) for general eigenforms were determined by Momose \cite{momose81} and Ribet \cite{ribet85} respectively. However, somewhat surprisingly, there does not seem to be a result analogous to (C) in the literature for general modular eigenforms $f$. The first aim of this paper is to fill this minor gap, by formulating and proving an analogue of (C) for general eigenforms; see \S \ref{sect:largeimage}, in particular Theorem \ref{thm:bigimageGL2}.
  
  The second aim of this paper is to extend these results to pairs of modular forms: given two modular forms $f, g$, how large is the image of the product representation $\rho_f \times \rho_g$? In \S \ref{sect:jointlargeimage} we formulate and prove analogues of (A)--(C) for this product representation. This extends earlier partial results due to Ribet \cite[\S 6]{ribet75} (who proves the analogue of (B) for pairs of modular forms of level 1, and sketches an analogue of (A)); and of Lei, Zerbes and the author \cite[\S 7.2.2]{LLZ14} (who prove an analogue of (B) for pairs of modular forms of weight 2). 
  
  These results can all be extended to the case of arbitrary finite collections $(f_1, \dots, f_n)$ of modular forms, but we give the proofs only in the case $n = 2$ to save notation.

  In section \ref{sect:specialelt}, we give an application of these results which was the original motivation for our study of images of Galois representations; this is to exhibit certain special elements in the images of the tensor product Galois representations $\rho_{f, p} \otimes \rho_{g, p}$ whose existence is important in Euler system theory. These results are used in \cite{KLZ15b} in order to prove finiteness results for Selmer groups.

 \subsection*{Acknowledgements}

  This paper was written while the author was a guest at the Mathematical Sciences Research Institute in Berkeley, California, for the Fall 2014 programme ``New Geometric Methods in Number Theory and Automorphic Forms''. The author would like to thank MSRI, for providing such a fantastic working environment; and the many other MSRI guests with whom he had illuminating conversations on this topic, notably Luis Dieulefait, Eknath Ghate, Kiran Kedlaya, Ken Ribet, and Sarah Zerbes. The original impetus for writing this paper came from an exchange on the website MathOverflow; the author would like to thank Jeremy Rouse for his valuable suggestions during this discussion.

 \section{Some profinite group theory}

  \subsection{Preliminary lemmas}

   \begin{lemma}[Ribet]
    \label{lemma:lifting}
    Let $p \ge 5$ be prime, and let $K_1, \dots, K_t$ be finite unramified extensions of $\Qp$, with rings of integers $\cO_1, \dots, \cO_t$ and residue fields $k_1, \dots, k_t$. Let $G$ be a closed subgroup of $\SL_2(\cO_1) \times \dots \times \SL_2(\cO_t)$ which surjects onto $\PSL_2(k_1) \times \dots \times \PSL_2(k_t)$. Then $G = \SL_2(\cO_1) \times \dots \times \SL_2(\cO_t)$.
   \end{lemma}

   \begin{proof}
    If we assume that $G$ surjects onto $\SL_2(k_1) \times \dots \times \SL_2(k_t)$ this is a special case of Theorem 2.1 of \cite{ribet75}. So it suffices to check that there is no proper subgroup of $\SL_2(k_1) \times \dots \times \SL_2(k_t)$ surjecting onto $\PSL_2(k_1) \times \dots \times \PSL_2(k_t)$. But this follows readily by induction from the case $t = 1$, which is Lemma IV.3.4.2 of \cite{serre68b}.
   \end{proof}

   \begin{lemma}[{cf.~\cite[Lemma IV.3.4.1]{serre68b}}]
    \label{lemma:groupsoccurring}
    Let $K$ be a finite extension of $\Qp$ for some prime $p$, and let $Y_1, Y_2$ be closed subgroups of $\GL_2(\cO_K)$ such that $Y_1 \triangleleft Y_2$ and $Y_2 / Y_1$ is a nonabelian finite simple group. Then $Y_2 / Y_1$ is isomorphic to one of the following groups:
    \begin{itemize}
     \item $\PSL_2(\FF)$, where $\FF$ is a finite field of characteristic $p$ such that $\# \FF \ge 4$;
     \item the alternating group $A_5$.
    \end{itemize}
   \end{lemma}

   \begin{proof}
    Since the kernel of $\GL_2(\cO_K) \to \PGL_2(k)$ is solvable, where $k$ is the residue field of $K$, we see that any such quotient $Y_2 / Y_1$ is in fact a subquotient of $\PSL_2(k)$. The result now follows from the determination of the subgroup structure of $\PSL_2(k)$, which is due to Dickson; cf.~\cite[Theorem 6.25]{suzuki}.
   \end{proof}

   \begin{lemma}
    If $k$ and $k'$ are any two finite fields of characteristic $\ge 5$ and $\phi: \PSL_2(k) \cong \PSL_2(k')$ is a group isomorphism, then $\phi$ is conjugate in $\PGL_2(k')$ to an isomorphism induced by a field isomorphism $k \cong k'$.
   \end{lemma}

   \begin{proof}
    Since the groups $\PSL_2(k)$ for finite fields $k$ of characteristic $\ge 5$ are nonisomorphic unless $k \cong k'$, it suffices to check that every group automorphism of $\PSL_2(k)$ is induced by conjugation in $\PGL_2(k)$, which is a standard fact.
   \end{proof}

   We also have an ``infinitesimal'' version of this statement, which we will use later in the paper.

   \begin{lemma}
    \label{lemma:liealgs}
    If $K$ and $K'$ are finite extensions of $\Qp$ for some prime $p$, $B$ and $B'$ are central simple algebras of degree 2 over $K$ and $K'$ respectively, and the Lie algebras $\operatorname{sl}_1(B)$ and $\operatorname{sl}_1(B')$ are isomorphic as Lie algebras over $\Qp$, then the isomorphism is induced by a field isomorphism $K \cong K'$ and an isomorphism of central simple algebras $B \cong B'$ over $K$.
   \end{lemma}

   \begin{proof}
    We may recover $K$ from $\operatorname{sl}_1(B)$ as the algebra of $\Qp$-endomorphisms of $\operatorname{sl}_1(B)$ commuting with the adjoint action of $\operatorname{sl}_1(B)$; thus it suffices to consider the case $K' = K$. There are exactly two central simple algebras of degree 2 over any $p$-adic field (one unramified and one unramified), and their Lie algebras are non-isomorphic; and every automorphism of either of these is inner (since the corresponding Dynkin diagram $A_1$ has no automorphisms).
   \end{proof}

  \subsection{Subgroups of adele groups}

   Let $F$ be a \emph{finite \'etale extension} of $\QQ$; that is, $F$ is a ring of the form $\bigoplus_{i = 1}^t F_i$, where $F_i$ are number fields.

   A \emph{quaternion algebra} over $F$ is defined in the obvious way: it is simply an $F$-algebra $B$ of the form $\bigoplus_{i = 1}^t B_i$, where $B_i$ is a quaternion algebra over $F_i$ (a central simple $F_i$-algebra of degree 2); we allow the case of the split algebra $M_{2 \times 2}(F_i)$. There is a natural norm map
   \[ \norm_{B/F}: B^\times \to F^\times\]
   (which is just the product of the reduced norm maps of the $B_i$ over $F_i$).

   We fix a homomorphism of algebraic groups $k: \mathbf{G}_m \to \operatorname{Res}_{F / \QQ} \mathbf{G}_m$; this just amounts to a choice of integers $(k_1, \dots, k_t)$ such that $k(\lambda) = (\lambda^{k_1}, \dots, \lambda^{k_t})$.

   \begin{definition}
    \label{def:G}
    For $B$, $F$, $k$ as above, we let $G$ and $G^\circ$ be the algebraic groups over $\QQ$ such that for any $\QQ$-algebra $R$ we have
    \[ G(R) = \{ (x, \lambda) \in (B \otimes R)^\times \times R^\times: \norm_{B/F}(x) = \lambda^{1-k} \},\]
    and
    \[ G^\circ(R) = \{ x \in (B \otimes R)^\times: \norm_{B/F}(x) = 1\}.\]
   \end{definition}

   Then $G$ and $G^\circ$ are linear algebraic groups over $\QQ$, and $G^\circ$ is naturally a subgroup of $G$. More generally, we may fix a maximal order $\cO_B$ in $B$ and thus define $G$ and $G^\circ$ as group schemes over $\ZZ$. For all but finitely many primes $p$, we then have
   \[ G(\Zp) = \{ (x, \lambda) \in \GL_2(\cO_F) \times \Zp^\times: \det(x) = \lambda^{1-k} \};\]
   changing the choice of $\cO_B$ does not change $G(\Zp)$ away from finitely many primes $p$.

   \begin{theorem}
    \label{thm:mainthm0}
    Let $U^\circ$ be a compact closed subgroup of $G^\circ(\hat\QQ)$, where $\hat\QQ = \QQ \otimes \hat\ZZ$ is the finite adeles of $\QQ$, such that:
    \begin{itemize}
     \item for every prime $p$, the projection of $U^\circ$ to $G^\circ(\Qp)$ is open in $G^\circ(\Qp)$;
     \item for all but finitely many primes $p$, the projection of $U^\circ$ to $G^\circ(\Qp)$ is $G^\circ(\Zp)$.
    \end{itemize}
    Then $U^\circ$ is open in $G^\circ(\hat\QQ)$.
   \end{theorem}

   The proof we shall give of this theorem is a relatively straightforward generalization of the case $F = \QQ$, $B = M_{2 \times 2}(\QQ)$, $k=2$, which is the Main Lemma of \cite[\S IV.3.1]{serre68b}.

   \begin{proof}
    Let $S$ be a finite set of primes containing 2, 3, all primes ramified in $F / \QQ$, all primes at which $B$ is ramified, and all primes $p$ such that the projection of $U^\circ$ to $G^\circ(\Qp)$ is not equal to $G^\circ(\Zp)$.

    For a prime $p$, let $k_p = \prod_{w \mid p} k_w$ where the product is over primes $w \mid p$ of $F$, and $k_w$ is the residue field of $F$ at $w$. Then for each $p \notin S$ we have a natural map
    \[ U^\circ \to \PSL_2(k_p)\]
    given by projection to the $p$-component and the natural quotient map. By the definition of $S$, it is surjective. I claim that the restriction of this map to $U^\circ \cap G^\circ(\Zp)$ is also surjective (where we consider $G^\circ(\Zp)$ as a subgroup of $G^\circ(\hat\QQ)$ in the natural way).

    If this is not the case, then the group
    \[ Q = U^\circ / \left(U^\circ \cap G^\circ(\Zp)) \right) \]
    must have a quotient isomorphic to a nontrivial quotient of $\PSL_2(k_p)$, and in particular this group must surject onto the simple group $\PSL_2(k)$ for some finite field $k$ of characteristic $p$. However, the group $Q$ is exactly the image of $U^\circ$ in $\prod_{q \ne p} G^\circ(\QQ_q)$. Hence the finite simple group $\PSL_2(k)$ must be a subquotient of an open compact subgroup of $\prod_{q \ne p} G^\circ(\QQ_q)$; but this is not possible, since $\prod_{q \ne p} G^\circ(\QQ_q)$ is a compact subgroup of $\prod_{q \ne p} \GL_2(L \otimes \QQ_q)$ for any \'etale extension $L/F$ which splits $B$, and thus is conjugate to a closed subgroup of the maximal compact subgroup $\prod_{q \ne p} \GL_2(\cO_L \otimes \ZZ_q)$, and we know that this group does not have $\PSL_2(k)$ as a quotient by Lemma \ref{lemma:groupsoccurring}. Hence $U^\circ \cap G^\circ(\Zp)$ is a subgroup of $G^\circ(\Zp) = \SL_2(\cO_F \otimes \Zp)$ which surjects onto $\PSL_2(k_p)$, and by Lemma \ref{lemma:lifting}, this implies that $G^\circ(\Zp) \subseteq U^\circ$.
    
    So we have $\prod_{p \notin S} G^\circ(\Zp) \subseteq U^\circ$. In order to show that $U^\circ$ is open in $G^\circ(\hat\QQ)$, it therefore suffices to show that the image of $U^\circ$ is open in $\prod_{p \in S} G^\circ(\Qp)$. However, since $G^\circ(\Qp)$ contains a finite-index pro-$p$ subgroup for each $p \in S$, and $S$ is finite, one sees easily by induction on $\#S$ that any subgroup of $\prod_{p \in S} G^\circ(\Qp)$ whose projection to $G^\circ(\Qp)$ is open for all $p \in S$ must itself be open.
   \end{proof}

   \begin{theorem}
    \label{thm:mainthm}
    Let $U$ be a compact subgroup of $G(\hat\QQ)$, where $\hat\QQ = \QQ \otimes \hat\ZZ$ is the finite adeles of $\QQ$, such that:
    \begin{itemize}
     \item for every prime $p$, the projection of $U$ to $G(\Qp)$ is open in $G(F \otimes \Qp)$;
     \item for all but finitely many primes $p$, the projection of $U$ to $G(\Qp)$ is $G(\Zp)$;
     \item the image of $U$ in $\hat\QQ^\times$ is open.
    \end{itemize}
    Then $U$ is open in $G(\hat\QQ)$.
   \end{theorem}

   \begin{proof}
    Let $U^\circ = U \cap G^\circ(\hat\QQ)$. We claim $U^\circ$ satisfies the hypotheses of the previous theorem.

    Since $G(\hat\QQ) / G^\circ(\hat\QQ) \cong \hat\QQ^\times$ is abelian, the group $U^\circ$ contains the closure of the commutator subgroup of $U$. Since $\SL_2(\cO_F \otimes \Zp)$ is the closure of its own commutator subgroup for $p \ge 5$, we see that if $p \ge 5$ and $U$ surjects onto $G(\Zp)$, then $U^\circ$ surjects onto $G^\circ(\Zp)$.

    Let us show that for an arbitrary prime $p$, the commutator subgroup of $G^\circ(\Zp)$ has finite index. It suffices to show the corresponding result for $\SL_1(\mathcal{O}_B)$ for $B$ a quaternion algebra (possibly split) over a $p$-adic field; and this is equivalent to the statement that the Lie algebra $\mathfrak{sl}_1(B)$ is a nontrivial simple Lie algebra, which is clear since it becomes isomorphic to $\mathfrak{sl}_2$ after base extension to any splitting field of $B$.

    By the previous theorem, we conclude that $U$ contains an open subgroup of $G^\circ(\hat\QQ)$. But the image of $U$ in $\hat\QQ^\times$ is open by hypothesis, so we conclude that $U$ is in fact open in $G(\hat\QQ)$.
   \end{proof}

   \begin{remark}
    We cannot dispense with the hypothesis that the image of $U$ in $\hat\QQ^\times$ is open: there exist proper closed subgroups of $\hat\ZZ^\times$ whose projection to $\ZZ_p^\times$ is open for all $p$, but which are not open in $\hat\ZZ^\times$, such as the group $\hat\ZZ^{\times 2}$. We may even arrange that the projection to $\ZZ_p^\times$ is surjective for all $p$, as with the group $\{ x : x_p \in \ZZ_p^{\times 2}\ \forall p > 2\} \cup \{ x : x_p \notin \ZZ_p^{\times 2} \ \forall p > 2\}$.
   \end{remark}

 \section{Large image results for one modular form}
  \label{sect:largeimage}
  
  \subsection{Setup}

   Let $f$ be a normalized cuspidal modular newform of weight $k \ge 2$, level $N$ and character $\varepsilon$. We write $L = \QQ(a_n(f): n \ge 1)$ for the number field generated by the $q$-expansion coefficients of $f$. Note that $L$ is totally real if $\varepsilon = 1$, and is a CM field if $\varepsilon \ne 1$.

   \begin{definition}\mbox{~}
    \begin{enumerate}
     \item For $p$ prime, we write
     \[ \rho_{f, p}: G_{\QQ} \to \GL_2(L \otimes \Qp) \]
     for the unique (up to isomorphism) representation satisfying $\operatorname{Tr} \rho_f(\sigma_\ell^{-1}) = a_\ell(f)$ for all $\ell \nmid Np$, where $\sigma_\ell$ is the arithmetic Frobenius.
     \item We write
     \[\rho_{f}: G_{\QQ} \to \GL_2(L \otimes \hat\QQ) \]
     for the product representation $\prod_p \rho_{f, p}$, where $\hat\QQ$ is the ring of finite adeles of $\QQ$.
    \end{enumerate}
   \end{definition}

   The condition (1) only determines $\rho_{f, p}$ up to conjugacy, and we can (and do) assume that its image is contained in $\GL_2(\cO_L \otimes \Zp)$, where $\cO_L$ is the ring of integers of $L$. Thus $\rho_f$ takes values in $\GL_2(\cO_L \otimes \widehat\ZZ)$, where $\widehat\ZZ = \prod_p \Zp$ is the profinite completion of $\ZZ$.

   \begin{remark}
    Our normalizations are such that if $f$ has weight 2, $\rho_{f, p}$ is the representation appearing in the \'etale cohomology of $X_1(N)$ with trivial coefficients. Some authors use an alternative convention that $\operatorname{Tr} \rho_f(\sigma_\ell) = a_\ell(f)$, which gives the representation appearing in the Tate module of the Jacobian $J_1(N)$; this is exactly the dual of the representation we study, so the difference between the two is unimportant when considering the image.
   \end{remark}

  \subsection{The theorems of Momose, Ribet, and Papier}

   For $\chi$ a Dirichlet character, we let $f \otimes \chi$ denote the unique newform such that $a_\ell(f \otimes \chi) = \chi(n) a_\ell(f)$ for all but finitely many primes $\ell$.

   \begin{definition}[{\cite[\S 3]{ribet85}}]
    An \emph{inner twist} of $f$ is a pair $(\gamma, \chi)$, where $\gamma: L \into \CC$ is an embedding and $\chi$ is a Dirichlet character, such that the conjugate newform $f^\gamma$ is equal to the twist $f \otimes \chi$.
   \end{definition}

   Note that we always have $\overline{f} = f \otimes \varepsilon^{-1}$, so any newform of non-trivial character has at least one nontrivial inner twist.

   Lemma 1.5 of \cite{momose81} shows that if $(\gamma, \chi)$ is an inner twist of $f$, then $\chi$ takes values in $L^\times$ and $\gamma(L) = L$. Thus the inner twists $(\gamma, \chi)$ of $f$ form a group $\Gamma$ with the group law
   \[ (\gamma, \chi) \cdot (\sigma, \mu) = (\gamma \cdot \sigma, \chi^{\sigma} \cdot \mu). \]
   Moreover, for any $(\gamma, \chi) \in \Gamma$, the conductor of $\chi$ divides $N$ if $N$ is odd, and divides $4N$ if $N$ is even.

   It is well known that if there exists a nontrivial $\chi$ such that $f \otimes \chi = f$, then $f$ must be of CM type and $\chi$ must be the quadratic Dirichlet character attached to the corresponding imaginary quadratic field.

   We now assume (until further notice) that $f$ is not of CM type. Thus, for any inner twist $(\gamma, \chi) \in \Gamma$, the Dirichlet character $\chi$ is uniquely determined by $f$ and $\gamma$, and we write it as $\chi_\gamma$. The map $(\gamma, \chi) \mapsto \gamma$ identifies $\Gamma$ with an abelian subgroup of $\operatorname{Aut}(L / \QQ)$; we write $F$ for the subfield of $L$ fixed by $\Gamma$. The extension $L / F$ is Galois, with Galois group $\Gamma$ \cite[Proposition 1.7]{momose81}.

   Let us write $H$ for the open subgroup of $G_\QQ$ which is the intersection of the kernels of the Dirichlet characters $\chi_\gamma$ for $\gamma \in \Gamma$, interpreted as characters of $G_{\QQ}$ in the usual way. Then for all $\sigma \in H$ we have $\operatorname{Tr} \rho_f(\sigma) \in F \otimes \hat\QQ$.

   \begin{theorem}[Momose, Ribet, Ghate--Gonzalez-Jimenez--Quer]
    There exists a central simple algebra $B$ of degree 2 over $F$, unramified outside $2 N \operatorname{disc}(L / \QQ) \infty$, and an embedding $B \into M_{2 \times 2}(L)$, with the following property: we have
    \[ \rho_{f}(H) \subseteq B(F \otimes \hat\QQ)^\times \subseteq \GL_2(L \otimes \hat\QQ).\]
    Moreover, for all but finitely many primes $p$ we have $B \otimes \Qp = M_{2 \times 2}(F \otimes \Qp)$, and we may conjugate $\rho_{f, p}$ such that
    \[ \rho_{f, p}(H) = \{ x \in \GL_2(\cO_{F} \otimes \Zp) : \operatorname{det} x \in \ZZ_p^{\times (k - 1)} \} \tag{\dag}.\]
   \end{theorem}

   \begin{proof}
    This result is mostly proved in \cite{ribet85}, building on earlier results of Momose \cite{momose81}; the only statement not covered there is the explicit bound on the set of primes at which $B$ may ramify, which is Corollary 4.7 of \cite{GhateGJQuer}.
   \end{proof}

   We will need later in the paper the following refinement:

   \begin{corollary}[Papier]\label{papier}
    Let $p$ be such that $B$ and $L$ are unramified above $p$, and $\rho_{f, p}(H)$ is the whole group $(\dag)$; and let $\sigma \in G_{\QQ(\mu_{p^\infty})}$. Then the image of the coset $\sigma \cdot \left(H \cap G_{\QQ(\mu_{p^\infty})}\right)$ under $\rho_{f, p}$ is the set
    \[ \tbt \alpha 0 0 {\varepsilon(\sigma) \alpha^{-1}}\SL_2(\cO_F \otimes \Zp),\]
    for any $\alpha \in (\cO_L \otimes \Zp)^\times$ such that $\gamma(\alpha) = \chi_\gamma(\sigma) \alpha$ for all $\gamma$ in $\Gal(L / F)$.
   \end{corollary}

   \begin{proof}
    See \cite[Theorem 4.1]{ribet85}. (Strictly speaking, Ribet in fact only shows that there is $\alpha \in L^\times$ with this property, and excludes any primes $p$ such that $\alpha$ is not a $p$-adic unit. However, since we have assumed $L / F$ is unramified above $p$, we can always re-scale $\alpha$ to be a $p$-adic unit.)
   \end{proof}

  \subsection{Adelic open image for $\GL_2$}

   Since the determinant of $\rho_f|_H$ is $\chi^{1-k}$, where $\chi: G_{\QQ} \to \hat\ZZ^\times$ is the adelic cyclotomic character, we can extend $\rho_f$ to a homomorphism $\tilde\rho_f: H \to G(\hat\QQ)$, where $G$ is the algebraic group of Definition \ref{def:G} (for the specific choices of $B$, $F$ and $k$ as in this section). This homomorphism is characterized by the requirement that its projection to $\GL_2(L \otimes \hat\QQ)$ is $\rho_f$, and its projection to $\hat\QQ^\times$ is the cyclotomic character.

   Applying Theorem \ref{thm:mainthm} to $\tilde\rho_f(H)$, we obtain the first new result of this paper:

   \begin{theorem}
    \label{thm:bigimageGL2}
    The image of $H$ under $\tilde\rho_f$ is an open subgroup of $G(\hat\QQ)$.
   \end{theorem}

   \begin{remark}
    One can show exactly the same result with modular forms replaced by Hilbert modular forms for a totally real field $E$, since the Momose--Ribet theorem has been generalized to this context by Nekov\'a\v{r} \cite[Theorem B.4.10]{nekovar12}. We have stated the result only for elliptic modular forms in order to save notation.
   \end{remark}

  \subsection{The CM case}

   For completeness, we briefly describe the image of $\tilde\rho_f$ in the CM case.
   
   Let us now suppose $f$ is of weight $k \ge 2$ and is of CM type, associated to some Hecke character
   \[ \psi: \hat{K}^\times \to \tilde{L}^\times\]
   for some imaginary quadratic field $K$ and Gr\"ossencharacter $\psi$ of infinity-type $(1-k, 0)$, with $\psi$ taking values in some extension $\tilde{L}$ of $L$. The relation between $f$ and $\psi$ is given by
   \[ a_p(f) = \psi(\varpi_\frp) + \psi(\varpi_{\frp'})\]
   whenever $p$ is a rational prime, not dividing the level of $f$, which splits in $K$ as $\frp \frp'$. Here $\varpi_{\frp} \in \widehat{K}^\times$ is a uniformizer at $\frp$.

   Let us write $\hat\psi$ for the homomorphism $K^\times \backslash \hat{K}^\times \to (\hat{K} \otimes_K \tilde{L})^\times$ defined by
   \[ \hat\psi(x) = x^{1-k} \psi(x). \]
   If we identify $K^\times \backslash \hat{K}^\times$ with $G_K^{\mathrm{ab}}$ via the Artin map\footnote{Normalized in the French manner, so geometric Frobenius elements correspond to uniformizers.}, then the adelic Galois representation $\rho_g$ is given by $\Ind_{G_K}^{G_{\QQ}} (\hat\psi)$.

   Note that there is a finite-index subgroup $U \subseteq \hat\cO_K^\times$ contained in the kernel of $\psi$; thus the image of $\hat\psi$ contains a finite-index subgroup of the group $\{ x^{1 - k} : x \in \hat\cO_K^\times\}$. In particular, for almost all primes $p$ the image of $G_K$ under $\rho_{g, p}$ contains the group
   \[ \left\{ \tbt{x^{1-k}}{0}{0}{\bar{x}^{1-k}} : x \in (\cO_K \otimes \Zp)^\times \right\}.\]

 \section{Joint large image}
   \label{sect:jointlargeimage}

  \subsection{Preliminaries}

   Now let $f$, $g$ be two newforms of weights $k_f, k_g \ge 2$, levels $N_f, N_g$ and characters $\varepsilon_f$ and $\varepsilon_g$, respectively. We assume neither $f$ nor $g$ is of CM type, and we write $L_f, L_g$ for their coefficient fields. We will need the following lemma:

   \begin{lemma}
    \label{lemma:ramakrishnan}
    Suppose there exist embeddings $L_f, L_g \into \CC$ such that we have
    \[ \frac{a_\ell(f)^2}{\ell^{k_f-1} \varepsilon_f(\ell)} = \frac{a_\ell(g)^2}{\ell^{k_g-1} \varepsilon_g(\ell)} \]
    for a set of primes $\ell$ of positive upper density. Then there is a Dirichlet character $\chi$ such that $g = f \otimes \chi$.
   \end{lemma}

   \begin{proof}
    This is a special case of Theorem A of \cite{ramakrishnan00b}.
   \end{proof}

   \begin{remark}
    Recall that the \emph{upper density} of a set of primes $S$ is defined by
    \[ \operatorname{UD}(S) = \limsup_{X \to \infty} \frac{\# \{ \ell \in S : \ell \le X\}}{\# \{ \ell : \ell \le X\}}.\]
    We will need below the easily-verified fact that if $S_1, \dots, S_n$ are sets of primes, then
    \[ \operatorname{UD}(S_1 \cup \dots \cup S_n) \le \operatorname{UD}(S_1) + \dots + \operatorname{UD}(S_n),\]
    so if $S_1 \cup \dots \cup S_n$ has positive upper density, then at least one of the sets $S_i$ has positive upper density.
   \end{remark}

   We can obviously apply the theory of the previous section to each of $f$ and $g$, and we use the subscripts $f, g$ to refer to the corresponding objects for each form; so we have number fields $F_f, F_g$, quaternion algebras $B_f, B_g$, and algebraic groups $G_f, G_g$.

   We may unify these as follows: we let $F = F_f \times F_g$, which is an \'etale extension of $\QQ$, and $B = B_f \times B_g$, which is a quaternion algebra over $F$; and the group $G^\circ$ of norm 1 elements of $G$ is just $G^\circ_f \times G^\circ_g$. We let
   \[ k: \mathbf{G}_m \to \operatorname{Res}_{F / \QQ} \mathbf{G}_m = \operatorname{Res}_{F_f / \QQ} \mathbf{G}_m \times \operatorname{Res}_{F_g / \QQ} \mathbf{G}_m\]
   be the character sending $\lambda$ to $(\lambda^{k_f}, \lambda^{k_g})$. Then Definition \ref{def:G} gives us an algebraic group
   \begin{align*}
    G &= \{ (x, \lambda) \in B^\times \times \mathbf{G}_m : \norm(x) = \lambda^{1-k}\}\\
    &= \{ (x_f, x_g, \lambda) \in B_f^\times \times B_g^\times \times \mathbf{G}_m: \norm(x_f) = \lambda^{1-k_f}, \norm(x_g) = \lambda^{1-k_g} \}.
   \end{align*}
   This is, of course, just the fibre product of $G_f$ and $G_g$ over $\mathbf{G}_m$. Letting $H = H_f \cap H_g$, we have a representation
   \[ \tilde\rho_{f, g}: G_{\QQ} \to G(\hat\QQ), \]
   and in particular
   \[ \tilde\rho_{f, g, p}: G_{\QQ} \to G(\Qp)\]
   for all primes $p$.

  \subsection{Big image for almost all $p$}

   \begin{proposition}
    \label{prop:bigimage2}
    Let $p \ge 5$ be a prime unramified in $B$, and let $U$ be a subgroup of $G(\Zp)$ which surjects onto $G_f(\Zp)$ and $G_g(\Zp)$. Then either $U = G(\Zp)$, or (after possibly conjugating $U$) there are primes $v \mid p$ of $\cO_{F_f}$, $w \mid p$ of $\cO_{F_g}$ and an isomorphism
    \[ \cO_{F_f, v} \cong  \cO_{F_g, w}, \]
    such that for all $(x, y, \lambda) \in U$ we have $ (y \bmod w) = \pm \lambda^{(k_f-k_g)/2}(x \bmod v)$.
   \end{proposition}

   \begin{proof}
    This is visibly a generalization of Proposition 7.2.8 of \cite{LLZ14}, and we follow essentially the same argument. (We have changed notation from $H$ to $U$ to avoid confusion with the Galois group $H$ above.)

    Let $U^\circ = U \cap G^\circ(\Zp)$. By the same commutator argument as before, $U^\circ$ is a subgroup of $G^\circ(\Zp) = G_f^\circ(\Zp) \times G_g^\circ(\Zp)$ which surjects onto either factor.

    By Goursat's Lemma, there are closed normal subgroups $N_f \triangleleft G_f^\circ(\Zp)$ and $N_g \triangleleft G_g^\circ(\Zp)$ such that $U^\circ$ is the graph of an isomorphism $\phi: G_f^\circ(\Zp) / N_f \cong G_g^\circ(\Zp)/N_g$.

    The maximal normal closed subgroups of $G_f^\circ(\Zp)$ are precisely the kernels of the quotient maps to $\PSL_2(k_v)$ for each prime $v \mid p$ of $F_f$, and every automorphism of $\PSL_2(k_v)$ is the composite of a field automorphism of $k_v$ and conjugation by an element of $\PGL_2(k_v)$. Hence, after possibly replacing $U$ by a conjugate of $U$ in $G(\Zp)$, we may find primes $v \mid p$ of $F_f$ and $w \mid p$ of $F_g$, and an isomorphism $\cO_{F_f, v} \cong \cO_{F_g, w}$, such that $U^\circ$ is contained in a conjugate of the group
    \[ \{ (x, y) \in G_f^\circ(\Zp) \times G_g^\circ(\Zp) : x \bmod v = \pm y \bmod w \}.\]

    For a general element $(x, y, \lambda) \in U_p$, let $t = (x \bmod v)^{-1} (y \bmod w) \in \GL_2(\FF)$, and let $[t]$ denote its image in $\GL_2(\FF) / \{\pm 1\}$. For any element $(u, v) \in U^\circ$, we have the same commutator identity as in \cite[Proposition 7.2.8]{LLZ14},
    \[ [u^{-1} t u] = [u^{-1} x^{-1} y u] = [x^{-1}][(xux^{-1})^{-1}(y vy^{-1})][y][v^{-1} u] = [x^{-1} y] = [t],\]
    since $(x u x^{-1}, y v y^{-1}) \in U^\circ$. This shows that $[t]$ commutes with every element of $\PSL_2(\FF)$, so that $t$ is a scalar matrix. It is clear that we must have $t^2 = \lambda^{k_f - k_g}$ by comparing determinants, and this gives the result.
   \end{proof}

   \begin{theorem}
    \label{thm:goodprimes}
    If $f$ is not Galois-conjugate to a twist of $g$, then for all but finitely many primes $p$ we have $\rho_{f, g, p}(H) = G(\Zp)$.
   \end{theorem}

   \begin{proof}
    Let us fix embeddings of $F_f$ and $F_g$ into $\CC$, and let $F$ be their composite.

    The above theorem shows that for all $p$ outside some finite set $S$, if $\rho_{f, g, p}(H) \ne G(\Zp)$, then there is some prime $v$ of $F$ above $p$ dividing the product
    \[ \prod_{\gamma \in \Gal(F_g / \QQ)} \left( a_\ell(f)^2 - \ell^{k_f-k_g} \gamma(a_\ell(g))^2\right) \]
    for all primes $\ell$ whose Frobenius elements lie in $H$. Since no nonzero element of $F$ may be divisible by infinitely many primes, we deduce that either $\rho_{f, g, p}(H) = G(\Zp)$ for all but finitely many $p$, or the above product is zero, so for each prime $\ell$ whose Frobenius lies in $H$, there is $\gamma \in \Gal(F_g / \QQ)$ (possibly depending on $\ell$) such that we have
    \[ \frac{a_\ell(f)^2}{\ell^{k-1} \varepsilon_f(\ell)} = \gamma\left(\frac{a_\ell(g)^2}{\ell^{k-1} \varepsilon_g(\ell)}\right) \]
    (since $\varepsilon_f(\ell) = \varepsilon_g(\ell) = 1$ for all such $\ell$). Since there are only finitely many possible $\gamma$, there must be at least one $\gamma \in \Gal(F_g / \QQ)$ such that the above equality holds for a set of $\ell$ of positive upper density. By Lemma \ref{lemma:ramakrishnan}, this implies that for some (and hence any) $\gamma' \in \Gal(L_g / \QQ)$ lifting $\gamma$, the conjugate form $g^\gamma$ is a twist of $f$.
   \end{proof}

  \subsection{Open image for all $p$}

   \begin{proposition}
    Let $p$ be arbitrary and let $U$ be a subgroup of $G(\Zp)$ which has open image in $G_f(\Zp)$ and $G_g(\Zp)$. Then either $U$ is open in $G(\Zp)$, or there are primes $v$ of $F_f$ and $w$ of $F_g$ above $p$, a field isomorphism $F_{f, v} \cong F_{g, w}$, and an isomorphism $B_f \otimes F_{f, v} \cong B_g \otimes F_{g, w}$, such that $U$ has a finite-index subgroup contained in a conjugate of the subgroup
    \[ \{ (x, y, \lambda) \in G(\Zp): y_w = \lambda^{(k_f-k_g)/2} x_v\} \]
    where $x_v$ and $y_w$ are the projections of $x$ and $y$ to the direct summands $(B_f \otimes F_{f, v})^\times$ and $(B_g \otimes F_{g, w})^\times$.
   \end{proposition}

   \begin{proof}
    This follows in a very similar way to Proposition \ref{prop:bigimage2} with all the groups concerned replaced by their Lie algebras. We know that $\mathfrak{u} = \operatorname{Lie}(U)$ is a subalgebra of $\operatorname{Lie}(G)$ which surjects onto $\operatorname{Lie}(G_f)$ and $\operatorname{Lie}(G_g)$. Since $G^\circ_f$ and $G^\circ_g$ are semi-simple we deduce that $\mathfrak{u}^\circ = \operatorname{Lie}(U^\circ)$ is a subgroup of $\operatorname{Lie}(G_f^\circ) \oplus \operatorname{Lie}(G_g^\circ)$ surjecting onto either factor. By Goursat's Lemma for Lie algebras, we deduce that it must be contained in the graph of an isomorphism between simple factors of $\operatorname{Lie}(G_f^\circ)$ and $\operatorname{Lie}(G_g^\circ)$. Using Lemma \ref{lemma:liealgs}, we deduce the above result.
   \end{proof}

   \begin{proposition}
    If $f$ is not Galois-conjugate to a twist of $g$, then $\rho_{f, g, p}(H)$ is open in $G(\Zp)$ for all primes $p$.
   \end{proposition}

   \begin{proof}
    By the previous result, if $\rho_{f, g, p}(H)$ is not open in $G(\Zp)$, there is an element $\gamma \in \Gal(F_f / \QQ)$ and a positive-density set of primes $\ell$ such that we have
    \[ \frac{a_\ell(f)^2}{\ell^{k-1} \varepsilon_f(\ell)} = \gamma\left(\frac{a_\ell(g)^2}{\ell^{k-1} \varepsilon_g(\ell)}\right). \]
    Ramakrishnan's theorem now tells us that $g^\gamma$ is a twist of $f$.
   \end{proof}

  \subsection{Adelic big image}

   \begin{theorem}
    Let $f$, $g$ be non-CM-type cusp forms of weights $k_f, k_g \ge 2$. Then either $\rho_{f, g}(H)$ is open in $G(\hat\QQ)$, or $k_f = k_g$ and $f$ is Galois-conjguate to a twist of $g$.
   \end{theorem}

   \begin{proof}
    Suppose $f$ is not Galois-conjguate to a twist of $g$. Then, by the results of the previous two sections, $\rho_{f, g}(H)$ is a compact subgroup of $G(\hat\QQ)$ whose image is open in $G(\Qp)$ for all primes $p$, and equal to $G(\Zp)$ for all but finitely many $p$. Applying Theorem \ref{thm:mainthm}, we deduce that this subgroup must be open in $G(\Zp)$.
   \end{proof}

   Via exactly the same methods and induction on $n$, one can prove the following generalization. We shall not give the proof here, as the notation becomes somewhat cumbersome, but the arguments are exactly as before:

   \begin{theorem}
    Let $f_1, \dots, f_n$ be newforms of weights $k_1, \dots, k_n \ge 2$. Then either
    \begin{itemize}
     \item there is a Dirichlet character $\chi$ and $i, j \in \{1, \dots, n\}$ such that $f_i \otimes \chi$ is Galois-conjugate to $f_j$, with $\chi \ne 1$ if $i = j$;
     \item or there is an open subgroup $H$ of $G_{\QQ}$ such that the image of $H$ under the map
     \[ \rho_{f_1} \times \dots \times \rho_{f_n} \times \chi: G_{\QQ} \to \GL_2(L_{f_1} \otimes \hat\QQ) \times \dots \times \GL_2(L_{f_1} \otimes \hat\QQ) \times \hat\QQ^\times\]
     is an open subgroup of $G(\hat\QQ)$, where $G$ is the algebraic group
     \[ \left\{ (g_1, \dots, g_n, \lambda) \in B_{f_1}^\times \times \dots \times B_{f_n}^\times \times \mathbf{G}_m: \norm(g_i) = \lambda^{1-k_i}\right\}.\]
    \end{itemize}
   \end{theorem}

   \begin{remark}
    Note that Serre \cite{serre91} has formulated a general conjecture on the image of Galois representations for motives: for any motive $M$ of rank $r$ over a number field $K$, one can define a connected subgroup $MT(M)$ of $\GL_r / \QQ$ such that the image of $\rho_{M}: G_{K} \to \GL_r(\hat\QQ)$ is contained in $MT(M)(\hat\QQ)$. Thus a finite-index subgroup $H$ of $G_{\QQ}$ lands in $MT^0(M)(\hat\QQ)$, where $MT^0(M)$ is the identity component.

    In general one does not expect $\rho_{M}(H)$ to be open in $MT^0(M)(\hat\QQ)$, because of obstructions arising from isogenies; e.g. if $M = \QQ(2)$, then $MT(M) = \mathbf{G}_m$, but the image of $G_{\QQ}$ is the group of squares in $\hat\ZZ^\times$, which is not open. However, there is a distinguished class of ``maximal'' motives for which this should be the case.

    The motive $M(f)$ attached to a weight $k$ modular form is not maximal if $k > 2$, but $M(f) \oplus \QQ(1)$ is maximal if $f$ is not of CM type (cf.~\S 11.10 of op.cit.), and the group $G_f$ is the connected component of $MT(M(f) \oplus \QQ(1))$. Thus we have verified Serre's open image conjecture for the maximal motives
    \[ M(f_1) \oplus \dots \oplus M(f_n) \oplus \QQ(1)\]
    whenever the $f_i$ are non-CM forms of weight $\ge 2$ and no $f_i$ is Galois-conjugate to a twist of $f_j$.
   \end{remark}

  \section{Special elements in the images}
   \label{sect:specialelt}
   
   \subsection{Setup}

    This section is more technical, and was the original motivation for the present work: to find elements in the images of $\rho_{f, p} \times \rho_{g, p}$ with certain special properties. In this section we fix newforms $f, g$ as before, and a Galois extension $L / \QQ$ with embeddings $L_f, L_g \into L$; we then have representations $\rho_{f, \frp}, \rho_{g, \frp}: G_{\QQ} \to \GL_2(\cO_{L, \frp})$ for each prime $\frp$ of $L$.

    Let $V_\frp$ be the four-dimensional $L_{\frp}$-vector-space $L_{\frp}^{\oplus 4}$, with $G_{\QQ}$ acting via the tensor-product Galois representation $\rho_{f, \frp} \otimes \rho_{g, \frp}$; and let $T_\frp$ be the $G_{\QQ}$-stable $\cO_{L, \frp}$-lattice $\cO_{L, \frp}^{\oplus 4}$ in $V_{\frp}$.

    Our aim is to verify the following conditions, in as many cases as possible:

    \begin{hypothesis}[$\operatorname{Hyp}(\QQ(\mu_{p^\infty}), V_{\frp})$]\mbox{~}
     \begin{enumerate}
      \item $V_{\frp}$ is an irreducible $L_{\frp}\left[G_{\QQ(\mu_{p^\infty})}\right]$-module (where $p$ is the rational prime below $\frp$).
      \item There is an element $\tau \in G_{\QQ(\mu_{p^\infty})}$ such that $V_{\frp} / (\tau - 1)V_{\frp}$ has dimension 1 over $L_{\frp}$.
     \end{enumerate}
    \end{hypothesis}

    \begin{hypothesis}[$\operatorname{Hyp}(\QQ(\mu_{p^\infty}), T_{\frp})$]\mbox{~}
     \begin{enumerate}
      \item $T_{\frp} \otimes k_{\frp}$ is an irreducible $k_{\frp}\left[G_{\QQ(\mu_{p^\infty})}\right]$-module, where $k_{\frp}$ is the residue field of $L_{\frp}$.
      \item There is an element $\tau \in G_{\QQ(\mu_{p^\infty})}$ such that $T_{\frp} / (\tau - 1) T_{\frp}$ is free of rank 1 over $\cO_{L, \frp}$.
     \end{enumerate}
    \end{hypothesis}

    Our formulation of these is exactly that of \cite[Chapter 2]{rubin00}. Note that $\operatorname{Hyp}(\QQ(\mu_{p^\infty}), T_{\frp}) \Rightarrow \operatorname{Hyp}(\QQ(\mu_{p^\infty}), V_{\frp})$. We note the following preliminary negative result:

    \begin{proposition}
     If $\varepsilon_f \varepsilon_g$ is the trivial character, then $\operatorname{Hyp}(\QQ(\mu_{p^\infty}), V_{\frp})$ is false (for every prime $\frp$).
    \end{proposition}

    \begin{proof}
     If $\varepsilon_f \varepsilon_g$ is trivial, the image of $G_{\QQ(\mu_{p^\infty})}$ under $\rho_{f, \frp} \times \rho_{g, \frp}$ is contained in the subgroup $\left\{ (x, y) \in \GL_2(L_\frp) \times \GL_2(L_\frp) : \det(xy) = 1\right\}$. An easy case-by-case check shows that the image of this subgroup under the tensor-product map to $\GL_4(L_\frp)$ contains no element $\tau$ such that $\tau - 1$ has one-dimensional cokernel.
    \end{proof}

   \subsection{Special elements: the higher-weight case}
    \label{sect:higherwt}

    In this section, we assume $f$ and $g$ have weights $\ge 2$, both $f$ and $g$ are non-CM, and $f$ is not Galois-conjugate to any twist of $g$.

    We say $\frp$ is a \emph{good prime} if the prime $p$ of $\QQ$ below $\frp$ is $\ge 5$, $p$ is unramified in the quaternion algebra $B$ over $F_f \oplus F_g$ described above, $p \nmid N_f N_g$, and the conclusion of Theorem \ref{thm:goodprimes} holds for $p$. For any good prime, it is clear that the irreducibility hypothesis (1) in $\operatorname{Hyp}(\QQ(\mu_{p^\infty}), T_{\frp})$ is satisfied.

    For convenience we set $N = \operatorname{LCM}(N_f, N_g)$ if $N_f$ and $N_g$ are both odd, and $N = 4\operatorname{LCM}(N_f, N_g)$ otherwise, so for any inner twist $(\gamma, \chi)$ of either $f$ or $g$, the conductor of $\chi$ divides $N$.

    \begin{proposition}
     \label{prop:existencetau}
     Let $u \in (\ZZ / N \ZZ)^\times$ be such that $\varepsilon_f(u) \varepsilon_g(u) \ne 1$. Let $\frp$ be a good prime, and suppose that $\chi_{\gamma}(u) = 1$ for all $\gamma$ in the decomposition group of $\frp$ in $\Gamma_f$, and similarly for $\Gamma_g$.

     Then $\Hyp(\QQ(\mu_{p^\infty}), V_{\frp})$ holds; and if $p \ge 7$ and $\varepsilon_f \varepsilon_g(u) \ne 1 \bmod p$, then in fact $\Hyp(\QQ(\mu_{p^\infty}), T_{\frp})$ holds.
    \end{proposition}

    \begin{proof}
     The condition on the decomposition groups implies that for $\sigma \in G_{\QQ(\mu_{p^\infty})}$ whose image in $(\ZZ / N_f N_g \ZZ)^\times$ is $u$, the quantities $\alpha$ arising in Papier's theorem (Corollary \ref{papier}) for $f$ and $g$ lie in $F_{f, \frp}$ and $F_{g, \frp}$ respectively, so we have $\rho_{f, \frp}(\sigma) \in \GL_2(F_{f, \frp})$ and $\rho_{g, \frp}(\sigma) \in \GL_2(F_{g, \frp})$. Since
     \[ (\rho_{f, \frp} \times \rho_{g, \frp})\left(H \cap G_{\QQ(\mu_{p^\infty})}\right) = \SL_2(\cO_{F_f, \frp}) \times \SL_2(\cO_{F_g, \frp}),\]
     it follows that the image of $G_{\QQ(\mu_{p^\infty})}$ under $\rho_{f, \frp} \times \rho_{g, \frp}$ contains the element
     \[ \left( \tbt x 0 0 {x^{-1} \varepsilon_f(u)}, \tbt y 0 0 {y^{-1} \varepsilon_g(u)} \right) \]
     for any $x \in \cO_{F_f, \frp}^\times$ and $y \in \cO_{F_g, \frp}^\times$. Choosing $x, y \in \ZZ_p^\times$ with $xy = 1$ and $x^{-2}\varepsilon_f(u) \ne 1$, $x^2 \varepsilon_g(u) \ne 1$ we see that the image of this element under the tensor product map is diagonal and has exactly one entry equal to 1, so $\Hyp(\QQ(\mu_{p^\infty}), V_{\frp})$ holds.

     If $p \ge 7$, then we may choose $x$ such that $x^{-2} \varepsilon_f(u) \ne 1$, $x^2 \varepsilon_g(u) \ne 1$ modulo $p$ (as there are at least three distinct quadratic residues modulo $p$); and the condition $\varepsilon_f \varepsilon_g(u) \ne 1 \bmod p$ implies that the fourth diagonal entry is also not equal to 1 modulo $p$. So $\Hyp(\QQ(\mu_{p^\infty}), T_{\frp})$ holds.
    \end{proof}

    \begin{remark}
     In particular, the proposition applies if $\varepsilon_f \varepsilon_g \ne 1$ and $F_{f, \frp} = L_{f, \frp}$ and $F_{g, \frp} = L_{g, \frp}$, since in this case both decomposition groups are trivial and we may take any $u$ with $\varepsilon_f \varepsilon_g(u) \ne 1$. See \cite[Proposition 7.2.18]{LLZ14}, which is the special case where $L_{f, \frp} = L_{g, \frp} = \Qp$.
    \end{remark}

    \begin{proposition}
     \label{prop:existencetauII}
     Suppose there exists $u \in (\ZZ / N \ZZ)^\times$ such that $\varepsilon_g(u) = -1$, but $\chi_\gamma(u) = 1$ for all $\gamma \in \Gamma_f$. Then for \emph{all} good primes $\frp$, $\Hyp(\QQ(\mu_{p^\infty}), T_{\frp})$ holds.
    \end{proposition}

    \begin{proof}
     Since $p \nmid N_f N_g$, we may find $\sigma \in G_{\QQ(\mu_{p^\infty})}$ mapping to $u$. By Papier's theorem (Corollary \ref{papier} above), the image of the coset $\sigma \cdot (H \cap G_{\QQ(\mu_{p^\infty})})$ under $\rho_{f, \frp} \times \rho_{g, \frp}$ is the set
     \[\left \{(x, y) : x \in \SL_2(\cO_{F_f, \frp}), y \in \stbt \alpha 0 0 {-\alpha^{-1}}\SL_2(\cO_{F_g, \frp})\right\},\]
     where $\alpha \in \cO_{L_g, \frp}^\times$ is any element such that $\gamma(\alpha) = \chi_{\gamma}(\sigma) \alpha$ for all inner twists $(\gamma, \chi_\gamma)$ of $g$ such that $\gamma$ lies in the decomposition group of $\frp$.

     However, the coset $\stbt \alpha 0 0 {-\alpha^{-1}}\SL_2(\cO_{F_g, \frp})$ contains $\stbt \alpha 0 0 {-\alpha^{-1}} \stbt 0 1 {-1} 0 = \stbt 0 \alpha {\alpha^{-1}} 0$. Since $\alpha$ is only defined up to multiplication by $\cO_{F, \frp}^\times$, we may assume that $\alpha^2 \ne 1 \bmod \frp$ (using the assumption that $p \ge 5$). Then the element $\stbt 0 \alpha {\alpha^{-1}} 0$ is conjugate in $\GL_2(\cO_{L_g, \frp})$ to $\stbt 1 0 0 {-1}$.

     Hence the group $G_{\QQ(\mu_{p^\infty})}$ contains an element $\tau$ whose image in $\GL_2(\cO_{L_f, \frp}) \times \GL_2(\cO_{L_g, \frp})$ is conjugate to $\left( \stbt 1 1 0 1, \stbt 1 0 0 {-1} \right)$, and this acts on $T_\frp$ with cokernel free of rank 1 as desired.
    \end{proof}

    \begin{remark}
     Note in particular that the hypotheses of the preceding proposition are satisfied if $g$ has odd weight, and either $N_f$ and $N_g$ are coprime, or $f$ has trivial character and no nontrivial inner twists.
    \end{remark}

%%%%%%%%%%%%%%%

  \subsection{Special elements: the CM case}

   We now suppose that $f$, $g$ both have weights $\ge 2$, as before, and $f$ is non-CM, but $g$ is CM, associated to a Gr\"ossencharacter $\psi$ of an imginary quadratic field $K$. Let $\tilde L_g$ be the extension of $L_g$ in which the values of $\psi$ lie, and let us suppose that our embedding $L_g \into L$ extends to an embedding $\tilde L_g \into L$.

   We let $H$ be an open subgroup of $G_K$, with $G_K / H$ abelian, such that $H \subseteq H_f$ and $\hat\psi(H) \subseteq (\hat\cO_K)^{\times(1 - k)}$. In this CM setting, we say a prime $\frp$ of $L$ (above some rational prime $p$) is \emph{good} if $p \nmid N_f N_g$, $\frp$ is unramified in $F_f$ and in the quaternion algebra $B_f$, the image of $H$ under $\rho_{f, p}$ contains $G_f(\Zp)$, and the image of $G_K$ under $\hat\psi_p$ contains $(\cO_K \otimes \Zp)^{\times(1-k)}$. Since $H$ is open in $G_\QQ$, all but finitely many primes $p$ are good, as before.

   \begin{proposition}
    Suppose there exists $u \in (\ZZ / N \ZZ)^\times$ such that $\varepsilon_f \varepsilon_g(u) \ne 1$ and $\varepsilon_K(u) = 1$, where $\varepsilon_K$ is the quadratic Dirichlet character attached to $K$.

    Let $\frp$ be a good prime such that $L_{f, \frp} = F_{f, \frp}$ and $\tilde{L}_{g, \frp} = \Qp$. Then $\Hyp(\QQ(\mu_{p^\infty}), V_{\frp})$ holds; and if $p \ge 7$ and $\varepsilon_f \varepsilon_g(u) \ne 1 \bmod p$, then in fact $\Hyp(\QQ(\mu_{p^\infty}), T_{\frp})$ holds.
   \end{proposition}

   \begin{proof}
    This is similar to Proposition \ref{prop:existencetau}. Since $\SL_2(\cO_{F, \frp})$ and $\ZZ_p^\times$ have no common quotient, the image of $H \cap G_{\QQ(\mu_{p^\infty})}$ under $\rho_{f, \frp} \times \rho_{g, \frp}$ is the whole of the group
    \[ \SL_2(\cO_{F, \frp}) \times \left\{ \tbt y 0 0 {y^{-1}}: y \in \ZZ_p^\times\right\}.\]
    If we choose $\sigma \in G_{\QQ(\mu_{p^\infty})}$ lifting $u$, then $\rho_{f, \frp}(\sigma) \in \GL_2(\cO_{F, \frp})$, and $\rho_{g, \frp}(\sigma)$ is diagonal; thus the image of the coset $\sigma \cdot(H \cap G_{\QQ(\mu_{p^\infty})})$ contains all elements of the form
    \[ \left( \tbt x 0 0 {x^{-1} \varepsilon_f(u)}, \tbt y 0 0 {y^{-1} \varepsilon_g(u)} \right) \]
    with $x \in \cO_{F_f, \frp}^\times$ and $y \in \ZZ_p^\times$. The proof now proceeds as before.
   \end{proof}

%%%%%%%%%%%%%%%

  \subsection{Special elements: the weight one case}
   \label{sect:wt1}

   We now assume $g$ is a weight 1 form, so the Galois representation $\rho_g$ lands in $\GL_2(L_g) \subset \GL_2(L_g \otimes \hat\QQ)$, and has finite image (i.e. it is an Artin representation). In this section we \emph{do} permit $g$ to be of CM type. As in the previous section, we assume that our other newform $f$ has weight $\ge 2$ and is not of CM type.

   \begin{theorem}
    Suppose $N_f$ is coprime to $N_g$. Then for all primes $\frp$ of $L$ such that $p \nmid N_g$ and $p$ is unramified in $F_f$ and $B_f$, we may find $\tau \in G_{\QQ(\mu_{p^\infty})}$ such that $V_\frp / (\tau - 1) V_\frp$ is 1-dimensional over $L_{\frp}$.

    For all but finitely many $\frp$, we may choose $\tau$ such that $T_\frp/(\tau-1)T_\frp$ is free of rank 1 over $\cO_{L, \frp}$.
   \end{theorem}

   \begin{proof}
    Let $p$ be the rational prime below $\frp$. As $\rho$ is unramified outside $N_g$ and $\rho_{f, \frp}$ is unramified outside $p N_f$, and $(p N_f, N_g) = 1$, we conclude that the splitting field of $\rho_g$ is linearly disjoint from that of $\rho_{f, \frp}$ and from $\QQ(\mu_{p^\infty})$. Hence, given any $a \in \rho_{g}(G_{\QQ})$ and $b \in \rho_{f, \frp}\left(G_{\QQ(\mu_{p^\infty})}\right)$, we may find $\tau \in G_{\QQ(\mu_{p^\infty})}$ such that $\rho_g(\tau) = a$ and $\rho_{f, \frp}(\tau) = b$.

    We know that $\rho_g$ is odd, so $\rho(G_{\QQ})$ contains an element $a$ conjugate to $\stbt{-1}001$.

    Meanwhile, since $f$ is not of CM type, $\rho_{f, \frp}\left(G_{\QQ(\mu_{p^\infty})}\right)$ contains a conjugate of an open subgroup of $\SL_2(F_{f, \frp})$, where $F_{f, \frp}$ is the fixed field of the extra twists of $f$ as in the previous section. In particular, it contains a conjugate of an open subgroup of $\SL_2(\Zp)$; so, after a suitable conjugation, the image contains the element $b = \stbt 1 {p^r} 0 1$ for $r \gg 0$. The preceding argument allows us to find $\tau \in G_{\QQ(\mu_{p^\infty})}$ such that $\rho_g(\tau) = a$ and $\rho_{f, \frp}(\tau) = b$. As $a \otimes b-1$ clearly has 1-dimensional kernel, we are done.

    For all but finitely many $\frp$ we have the stronger result that $\rho_{f,\frp}\left(G_{\QQ(\mu_{p^\infty})}\right)$ contains a conjugate of $\SL_2(\cO_{F, \frp})$, so we may take $r = 0$ and we deduce that $a \otimes b-1$ has 1-dimensional kernel modulo $p$.
   \end{proof}

    \begin{remark}
     Further strengthenings of the results of this section may be possible: it seems reasonable to expect that whenever $\varepsilon_f \varepsilon_g$ is nontrivial, $\Hyp(\QQ(\mu_{p^\infty}), T_{\frp})$ should hold for all but finitely many $\frp$. But I have not been able to prove this.
    \end{remark}

\providecommand{\bysame}{\leavevmode\hbox to3em{\hrulefill}\thinspace}
\providecommand{\MR}[1]{}
\renewcommand{\MR}[1]{%
 MR \href{http://www.ams.org/mathscinet-getitem?mr=#1}{#1}.
}
\newcommand{\articlehref}[2]{#2 (\href{#1}{link})}

 \end{document}